\documentclass[11pt]{amsart}
\pdfoutput=1
\usepackage[english]{babel}

\usepackage{amsthm}
\usepackage[mathcal]{euscript}

\usepackage{accents}

\newtheorem{thm}{Theorem}[section]
\newtheorem{lemma}[thm]{Lemma}
\newtheorem{prop}[thm]{Proposition}
\newtheorem{cor}[thm]{Corollary}
\theoremstyle{definition}
\newtheorem{ex}[thm]{Example}
\newtheorem{defi}[thm]{Definition}

\newtheorem{remark}[thm]{Remark}
\newtheorem*{thm*}{Theorem}

\newcommand{\munderbar}[1]{\underaccent{\bar}{#1}}

\newcommand{\gl}{\mathfrak{gl}}
\newcommand{\Sl}{\mathfrak{sl}}
\newcommand{\supp}{\mathrm{supp}}
\newcommand{\End}{\operatorname{End}}
\newcommand{\trl}{\sharp}
\newcommand{\tr}{\operatorname{tr}}
\newcommand{\id}{\mathrm{id}}
\newcommand{\frg}{\mathfrak{g}}
\newcommand{\frr}{\mathfrak{r}}
\newcommand{\frs}{\mathfrak{s}}
\newcommand{\bz}{{\bar{0}}}
\newcommand{\bo}{{\bar{1}}}

\numberwithin{equation}{section}

\begin{document}


\author[Bahturin, Y.]{Yuri Bahturin}
\address{Department of Mathematics and Statistics,
  Memorial University of Newfoundland,
  St. John's, NL, A1C5S7, Canada}
\email{bahturin@mun.ca}

\author[Dos Santos, H.S.]{Helen Samara Dos Santos}
\email{helensds@mun.ca}

\author[Hornhardt, C.]{Caio De Naday Hornhardt}
\email{cdnh22@mun.ca}

\author[Kochetov, M.]{Mikhail Kochetov}
\email{mikhail@mun.ca}


\thanks{The authors acknowledge financial support by the Natural Sciences and Engineering Research Council (NSERC) of Canada.}

\date{}

\title{Group gradings on the Lie and Jordan superalgebras $Q(n)$}

\subjclass[2010]{Primary 17B70; Secondary 17A70, 17C70, 16W50}


\keywords{Graded algebra, group grading, simple Lie superalgebra, classical Lie superalgebra, queer Lie superalgebra, simple Jordan superalgebra}

\begin{abstract}
We classify gradings by arbitrary abelian groups on the classical simple Lie and Jordan superalgebras $Q(n)$, $n \geq 2$, over an algebraically closed field of characteristic different from $2$ (and not dividing $n+1$ in the Lie case): fine gradings up to equivalence and $G$-gradings, for a fixed group $G$, up to isomorphism.
\end{abstract}

\maketitle

\section{Introduction}\label{se:intro}

The classification of gradings by arbitrary abelian groups on finite-di\-men\-si\-onal simple Lie and Jordan algebras over an algebraically closed field $\mathbb{F}$ of characteristic $0$ is essentially complete (see, e.g., the monograph \cite{EKmon} and the references therein). Descriptions of group gradings on classical simple Lie algebras (except $D_4$) were obtained in \cite{BSZ05} and \cite{BZ06}. In these descriptions, though, the same grading could be written in many ways. A classification of group gradings for classical simple Lie algebras was done in \cite{BK10} ($G$-gradings, for a fixed group $G$, up to isomorphism, with $\mathrm{char}\,\mathbb{F}\ne 2$) and \cite{Eld10} (fine gradings up to equivalence, with $\mathrm{char}\,\mathbb{F}=0$). Assuming the grading group $G$ finite and 
$\mathrm{char}\,\mathbb{F}=0$, a classification of inner $G$-gradings for type $A_n$ was obtained independently in \cite{ZZ10}. A method to relate gradings on associative algebras to gradings on corresponding Lie and Jordan algebras, which works in arbitrary characteristic and for sufficiently high dimension (including infinite), was proposed in \cite{BB,BBS} and used to classify gradings on finitary simple Lie algebras in \cite{BBK}.

We are interested in group gradings on finite-dimensional simple Lie and Jordan superalgebras. In characteristic $0$, these superalgebras (that are not Lie or Jordan algebras) were classified by V.~G.~Kac in \cite{artigokac,kacZ} (see also \cite{livrosuperalgebra,CM}). 
The classification of gradings on simple Lie superalgebras by the group $\mathbb{Z}$ was obtained in \cite{kacZ} and by finite cyclic groups in \cite{serganova}. More recently, fine gradings on the exceptional simple Lie and Jordan superalgebras were classified in \cite{artigoelduque} and \cite{CDM}, respectively. 

In the present work we consider the simple Lie and Jordan superalgebras $Q(n)$, $n\ge 2$. We reduce the problem of classifying group gradings on $Q(n)$ to the same problem for the simple Lie algebra of type $A_n$, which is the even component of the Lie superalgebra $Q(n)$, or to the simple Jordan algebra $M_n^{(+)}$, which is the even component of the Jordan superalgebra $Q(n)$. More precisely, we prove that a $G$-grading on $Q(n)$ is completely determined by a $G$-grading on its even component and an element of the group $G$ (Theorem \ref{ida}). This allows us to classify $G$-gradings up to isomorphism (Theorem \ref{th:isomorphism}) and fine gradings up to equivalence (Theorem \ref{th:fine-gradings}). For these results, we will need to assume that $\mathbb{F}$ is algebraically closed, $\mathrm{char}\,\mathbb{F}\ne 2$ and, in addition, $\mathrm{char}\,\mathbb{F}$ does not divide $n+1$ in the Lie case. Some of the auxiliary and related results are valid under weaker assumptions on $\mathbb{F}$, for example, the classification of gradings on the simple associative superalgebra that gives rise to $Q(n)$ (see Theorem \ref{th:assoc}).

We will now recall the necessary background on group gradings and fix the notation. All vector spaces, algebras and modules are over a fixed ground field $\mathbb{F}$ and usually assumed finite-dimensional. The components of the $\mathbb{Z}_2$-grading that is a part of the definition of a superalgebra will be labeled by superscripts $\bz$ and $\bo$, reserving subscripts for the components of other gradings (see Definition \ref{G-grading}). The degree according to the canonical $\mathbb{Z}_2$-grading will be referred to as parity (even or odd). A subspace $W$ of a superalgebra $A=A^\bz\oplus A^\bo$ is said to be \emph{compatible with the superalgebra structure} if $W=W^\bz\oplus W^\bo$ where $W^\bz=W\cap A^\bz$ and $W^\bo=W\cap A^\bo$. All subalgebras and ideals under consideration will be assumed to have this property. A homomorphisms of superalgebras is a homomorphism of algebras that preserves parity. 

Now we fix a group $G$, written multiplicatively, with the identity element denoted by $e$. 

\begin{defi}\label{G-grading}
A \emph{$G$-grading} on a superalgebra $A$ is a vector space decomposition $\Gamma:\,A = \bigoplus_{g \in G} A_g$ such that $A_g A_h\subseteq A_{gh}$, for all $g,h\in G$, and each $A_g$ is compatible with the superalgebra structure, i.e., $A_g=A_g^\bz\oplus A_g^\bo$. If $\Gamma$ is fixed, $A$ is referred to as a {\em $G$-graded superalgebra}. The nonzero elements $x\in A_g$ are said to be {\em homogeneous of degree $g$}. The {\em support} of $\Gamma$ is the set $\supp(\Gamma)=\{g\in G\;|\;A_g\neq 0\}$. We have $\supp(\Gamma)=\supp_\bz(\Gamma)\cup\supp_\bo(\Gamma)$ where $\supp_i(\Gamma)=\{g\in G\;|\;A_g^i\neq 0\}$.
\end{defi}

There is a concept of grading not involving groups. This is a decomposition $\Gamma:\,A=\bigoplus_{s\in S}A_s$ as a direct  sum of nonzero subspaces (compatible with the superalgebra structure) indexed by a set $S$ and having the property that, for any $s_1,s_2\in S$ with $A_{s_1}A_{s_2}\ne 0$, there exists (unique) $s_3\in S$ such that $A_{s_1}A_{s_2}\subseteq A_{s_3}$. For such a decomposition $\Gamma$, there may or may not exist a group $G$ containing $S$ that makes $\Gamma$ a $G$-grading. If such a group exists, $\Gamma$ is said to be a {\em group grading}. However, $G$ is usually not unique even if we require that it should be generated by $S$. The {\em universal grading group} is generated by $S$ and has the defining relations $s_1s_2=s_3$ for all $s_1,s_2,s_3\in S$ such that $0\ne A_{s_1}A_{s_2}\subseteq A_{s_3}$ (see \cite[Chapter 1]{EKmon} for details).

\begin{defi}
Let $\Gamma:\,A=\bigoplus_{g\in G} A_g$ and $\Delta:\,B=\bigoplus_{h\in H} B_h$ be two group gradings, with supports $S$ and $T$, respectively.
We say that $\Gamma$ and $\Delta$ are {\em equivalent} if there exists an isomorphism of superalgebras $\varphi\colon A\to B$ and a bijection $\alpha\colon S\to T$ such that $\varphi(A_s)=B_{\alpha(s)}$ for all $s\in S$. If $G$ and $H$ are universal grading groups then $\alpha$ extends to an isomorphism $G\to H$. In the case $G=H$, the $G$-gradings $\Gamma$ and $\Delta$ are {\em isomorphic} if $A$ and $B$ are isomorphic as $G$-graded superalgebras, i.e., if there exists an isomorphism of superalgebras $\varphi\colon A\to B$ such that $\varphi(A_g)=B_g$ for all $g\in G$. 
\end{defi}

If $\Gamma:\,A=\bigoplus_{g\in G} A_g$ and $\Gamma':\,A=\bigoplus_{h\in H} A'_h$ are two gradings on the same superalgebra, with supports $S$ and $T$, respectively, then we will say that $\Gamma'$ is a {\em refinement} of $\Gamma$ (or $\Gamma$ is a {\em coarsening} of $\Gamma'$) if for any $t\in T$ there exists (unique) $s\in S$ such that $A'_t\subseteq A_s$. If, moreover, $A'_t\ne A_s$ for at least one $t\in T$, then the refinement is said to be {\em proper}. A grading $\Gamma$ is said to be {\em fine} if it does not admit any proper refinements. Note that $A=\bigoplus_{(g,i)\in G\times\mathbb{Z}_2}A_g^i$ is a refinement of $\Gamma$. It follows that if $\Gamma$ is fine then the sets $\supp_\bz(\Gamma)$ and $\supp_\bo(\Gamma)$ are disjoint.

Given a $G$-grading $\Gamma:\,A=\bigoplus_{g\in G} A_g$, any group homomorphism $\alpha\colon G\to H$ induces an $H$-grading ${}^\alpha\Gamma$ on $A$ whose homogeneous component of degree $h$ is the sum of all $A_g$ with $\alpha(g)=h$. Clearly, ${}^\alpha\Gamma$ is a coarsening of $\Gamma$ (not necessarily proper). If $G$ is the universal group of $\Gamma$ then every coarsening of $\Gamma$ is obtained in this way. If $\Gamma$ and $\Gamma'$ are two gradings, with universal groups $G$ and $H$, then $\Gamma'$ is equivalent to $\Gamma$ if and only if $\Gamma'$ is isomorphic to ${}^\alpha\Gamma$ for some group isomorphism $\alpha\colon G\to H$. 

It can be shown that if $\Gamma$ is a group grading on a simple Lie superalgebra then the subgroup generated by $\supp(\Gamma)$ is abelian (the proof of Proposition 1.12 for Lie algebras in \cite{EKmon} works with minor changes). 
This result does not hold for simple Jordan algebras, as was observed in \cite{BSBF} using the Jordan algebras of bilinear forms. Nevertheless, from now on, we will work exclusively with \emph{abelian} grading groups. We denote by $\widehat{G}$ the group of characters of an abelian group $G$, i.e., homomorphisms $G\to\mathbb{F}^\times$. A $G$-grading $\Gamma$ on $A$ gives rise to an action of $\widehat{G}$ by automorphisms of $A$, $\eta_\Gamma\colon\widehat{G}\to\mathrm{Aut}(A)$, defined by $\eta_\Gamma(\chi)(a)=\chi(g)a$ for all $\chi\in\widehat{G}$, $g\in G$, $a\in A_g$. If $\mathbb{F}$ is algebraically closed and $\mathrm{char}\,\mathbb{F}=0$, the grading $\Gamma$ can be recovered from $\eta_\Gamma$ as the eigenspace decomposition of $A$ relative to the commuting automorphisms $\eta_\Gamma(\chi)$, $\chi\in\widehat{G}$.

\section{Lie and Jordan superalgebras $Q(n)$}\label{se:the-lie-superalgebra-Qn}

All classical series of Lie and Jordan superalgebras have standard matrix models. Let $M_{n\times m}$ be the space of $n\times m$ matrices over $\mathbb{F}$, 
$\mathrm{char}\,\mathbb{F}\ne 2$, and let $M_n=M_{n\times n}$.
The associative superalgebra $M(m,n)$ is the matrix algebra $M_{m+n}$ equipped with the following $\mathbb{Z}_2$-grading:

\begin{align*}
M(m,n)^\bz &=  \left\{
\left[
\begin{matrix}  
  a & 0 \\
  0 & b \\  
  \end{matrix}
\right] \in M(m,n) \; \Big| \;
a \, \in M_m \, \text{and} \, b \in M_n \right\},
\\
M(m,n)^\bo &=  \left\{
\left[
\begin{matrix}  
  0 & c \\
  d & 0 \\  
  \end{matrix}
\right] \in M(m,n) \; \Big| \;
c\in M_{m\times n} \, \text{and} \, d\in M_{n \times m} \right\}.
\end{align*}

Recall that any associative superalgebra $A=A^\bz\oplus A^\bo$ becomes a Lie superalgebra with respect to the \emph{supercommutator}, which is defined by 
\[
[x,y] = xy - (-1)^{ij} yx,\quad x\in A^i,\,y\in A^j,\,i,j\in\mathbb{Z}_2,
\]
for homogeneous elements and extended by linearity. This Lie superalgebra is denoted by $A^{(-)}$.

Similarly, any associative superalgebra $A$ becomes a Jordan superalgebra, denoted $A^{(+)}$, with respect to the \emph{supersymmetrized product}, which is defined by 
\[
x\circ y = xy + (-1)^{ij} yx,\quad x\in A^i,\,y\in A^j,\,i,j\in\mathbb{Z}_2.
\]
(This product is sometimes normalized with the factor $\frac12$.)

\subsection{The Lie superalgebra $Q(n)$}

The special linear Lie superalgebras (series $A$) are constructed from $M(m,n)^{(-)}$ by taking the quotient of the derived superalgebra modulo its center. The three orthosymplectic series ($B$, $C$, $D$) and the periplectic series ($P$) are constructed in the same way from the subalgebras of skew-symmetric elements in $M(m,n)^{(-)}$ with respect to appropriate superinvolutions. Series $Q$ is different in that we have to start from an associative superalgebra that is simple as a superalgebra but not as an algebra. Namely, let $A=R\times R$, with component-wise product, where $R=M_{n+1}$, $n\ge 1$. Then $(x,y)\mapsto(y,x)$ is an automorphism of $A$ of order $2$ and hence its eigenspace decomposition is a $\mathbb{Z}_2$-grading on $A$, with $A^\bz=\{(x,x)\;|\;x\in R\}$ and $A^\bo=\{(x,-x)\;|\;x\in R\}$. Note that $A^\bz$ is isomorphic to $R$ as an algebra and $A^\bo=uA^\bz$ where $u=(1,-1)$, so we may write $A=R\oplus uR$ where $u$ is odd, commutes with the elements of $R$ and satisfies $u^2=1$. This latter definition works even if $\mathrm{char}\,\mathbb{F}=2$. The associative superalgebra $A$ can be identified with a subalgebra of $M(n+1,n+1)$ as follows:
\[
\left\{
\left[
\begin{matrix}  
  a & b \\
  b & a \\  
  \end{matrix}
\right] \in M(n+1,n+1) \; \Big| \;
a,b \in M_{n+1} \right\}\stackrel{\sim}{\to}A,\quad 
\left[\begin{matrix}  
  a & b \\
  b & a \\  
  \end{matrix}\right]\mapsto a+ub.
\] 
Let $\widetilde{Q}(n)$ be the derived superalgebra of $A^{(-)}$. Then 
\[
\widetilde{Q} (n) = \left\{
\left[
\begin{matrix}  
  a & b \\
  b & a \\  
  \end{matrix}
\right] \in M(n+1,n+1) \; \Big| \;
a,b \in M_{n+1},\,\tr(b)=0 \right\}.
\] 
Set $Q(n)$ to be the quotient of $\widetilde{Q}(n)$ by its center, which is spanned by the identity matrix:
\[
Q(n)=\frac{\widetilde{Q} (n)}{\mathbb{F}1}.
\]
Thus, the even part of $Q(n)$ can be identified with $\mathfrak{pgl}(n+1):=R/\mathbb{F}1$ and the odd part with $\Sl(n+1)$. If $n\geq 2$ then $Q(n)$ is a simple Lie superalgebra (see \cite{artigokac, livrosuperalgebra}).
We may denote $Q(n)$ simply by $Q$ when $n$ is clear from the context.

If $\mathrm{char}\,\mathbb{F}$ does not divide $n+1$, the Lie algebras $\mathfrak{pgl}(n+1)$ and $\Sl(n+1)$ are isomorphic by means of the map $a+\mathbb{F}1\mapsto a^\trl$ where 
\[
a^\trl:=a-\frac{1}{n+1}\tr(a)1,\, a\in M_{n+1}.
\]
We may, therefore, identify both even and odd parts of the Lie superalgebra $Q(n)$ with $\mathfrak{sl}(n+1)$. 
In this way, we obtain another realization of $Q(n)$, which will be convenient for us in Section \ref{se:gradings-on-Qn}: 
\begin{equation}\label{identification}
  \begin{matrix}
    Q(n) & \stackrel{\sim}{\to} & \mathfrak{sl}(n+1)\oplus \mathfrak{sl}(n+1) \\
   \left[\begin{matrix}  a & b \\  b & a   \end{matrix}\right] + \mathbb{F}1 & \mapsto & (a^\trl, b).
 \end{matrix}
\end{equation}
To distinguish between $Q^\bz$ and $Q^\bo$, we will denote $(x,0)$ by $x$ and $(0,x)$ by $\munderbar{x}$, for $x\in\Sl(n+1)$. The mapping $x\mapsto \munderbar{x}$ is an isomorphism $Q^\bz\to Q^\bo$ as $Q^\bz$-modules, which looks as follows in terms of the realization of $Q(n)$ as a subalgebra of $M(n+1,n+1)^{(-)}/\mathbb{F}1$:
\[
  \left[ \begin{matrix}
  a & 0 \\
  0 & a
  \end{matrix} \right] + \mathbb{F}1 \mapsto \left[ \begin{matrix}
  0 & a^\trl \\
  a^\trl & 0
  \end{matrix}  
  \right] + \mathbb{F}1.
\]
It follows that the bracket of the Lie superalgebra $Q(n)$ in the realization \eqref{identification} is given by
\begin{equation}\label{star}
[a,b]=  ab-ba,\quad
[a,\munderbar{b}]  =  \underline{ab-ba},\quad
[\munderbar{a},\munderbar{b}]  = (ab+ba)^\trl,
\end{equation}
for all $a,b\in\mathfrak{sl}(n+1)$, where juxtaposition denotes multiplication in $M_{n+1}$.

\subsection{The Jordan superalgebra $Q(n)$}

The Jordan case is easier: set $Q(n)=A^{(+)}$ where $A=M_n\oplus uM_n$ (note that the matrix size is one less than in the Lie case). In other words, 
\[
Q (n) = \left\{
\left[
\begin{matrix}  
  a & b \\
  b & a \\  
  \end{matrix}
\right] \in M(n,n) \; \Big| \;
a,b \in M_{n}\, \right\}.
\] 
This is a simple Jordan superalgebra if $n\ge 2$.

We may identify both even and odd parts of the Jordan superalgebra $Q(n)$ with $M_n^{(+)}$, leading to another realization of $Q(n)$, which will be convenient in Section \ref{se:gradings-on-Qn}: 
\begin{equation}\label{J_identification}
  \begin{matrix}
    Q(n) & \stackrel{\sim}{\to} & M_n^{(+)}\oplus M_n^{(+)} \\
   \left[\begin{matrix}  a & b \\  b & a   \end{matrix}\right] & \mapsto & (a, b).
 \end{matrix}
\end{equation}
We will denote $(x,0)$ by $x$ and $(0,x)$ by $\munderbar{x}$, for $x\in M_n$, so the mapping $x\mapsto \munderbar{x}$ is an isomorphism $Q^\bz\to Q^\bo$ as $Q^\bz$-modules, which looks as follows in terms of the realization of $Q(n)$ as a subalgebra of $M(n,n)^{(+)}$:
\[
  \left[ \begin{matrix}
  a & 0 \\
  0 & a
  \end{matrix} \right]\mapsto \left[ \begin{matrix}
  0 & a \\
  a & 0
  \end{matrix}  
  \right].
\]
The product of $Q(n)$ in the realization \eqref{J_identification} is given by
\begin{equation}\label{J_star}
a\circ b  =  ab+ba,\quad
a\circ\munderbar{b} =  \underline{ab+ba},\quad
\munderbar{a}\circ\munderbar{b}  = ab-ba,
\end{equation}
for all $a,b\in M_n$, where juxtaposition denotes multiplication in $M_{n}$.

\subsection{Preliminary results}

We will later need the following fact about automorphisms of the superalgebra $Q=Q(n)$, $n\ge 2$. 
For an automorphism $\varphi\colon Q\to Q$, we will denote by $\varphi_i$ the restriction of $\varphi$ to $Q^i$, $i\in\mathbb{Z}_2$. Let $\upsilon$ be the parity automorphism, i.e., $\upsilon_i=(-1)^i\id$.

\begin{prop}\label{Aut-Qn}
If $\sqrt{-1}\in\mathbb{F}$ then the restriction map $\mathrm{Aut}(Q)\to\mathrm{Aut}(Q^\bz)$ is a surjective homomorphism whose kernel is generated by $\upsilon$.
\end{prop}

\begin{proof}
Suppose $\varphi\in\mathrm{Aut}(Q)$ belongs to the kernel of the restriction map, i.e., $\varphi_\bz=\id$. 
Then $\varphi_\bo$ is an automorphism of $Q^\bo$ as a $Q^\bz$-module, hence $\varphi_\bo=\lambda\,\id$ by Schur's Lemma, where $\lambda$ is a nonzero scalar. 
Since the composition $Q^\bo\times Q^\bo\to Q^\bz$ is nonzero and preserved by $\varphi$, it follows that $\lambda=\pm 1$. Therefore, $\varphi=\id$ or $\varphi=\upsilon$.

It is well known that the automorphism group of $Q^\bz$, which is $\mathfrak{sl}(n+1)\subset M_{n+1}^{(-)}$ in the Lie case and $M_n^{(+)}$ in the Jordan case, is generated by inner automorphisms (i.e., conjugations by invertible elements from the respective matrix algebra) and by an outer automorphism $\theta$, which is $\theta(x) = -x^{t}$ in the Lie case and $\theta(x) = x^{t}$ in the Jordan case. Thus, to prove the surjectivity of the restriction map, we only need to show that the inner automorphisms of $Q^\bz$ and $\theta$ can, indeed, be extended to automorphisms of the whole of $Q$. 
  
In the first case, if $\psi_r(x)=rxr^{-1}$ is an inner automorphism of $Q^\bz$ then we set $\varphi(a + \munderbar{b}) = \psi_r (a) + \underline{\psi_r (b)}$, for all $a,b \in Q^\bz$.
In the second case, we set $\varphi (a + \munderbar{b}) = \theta (a) + \sqrt{-1} \, \underline{\theta (b)}$, for all $a,b \in Q^\bz$.
For both cases, it is straightforward to verify that $\varphi$ is an automorphism. 
\end{proof}

\begin{remark}\label{extensions-of-automorphisms}
The above extensions of inner automorphisms of $Q^\bz$ are the inner automorphisms of $Q$. The extension of $\theta$ has order $4$, as its square equals $\upsilon$. It generates the group of outer automorphisms of $Q$. It follows that $\mathrm{Aut}(Q)$ is isomorphic to the semidirect product of the group of inner automorphisms of $Q^\bz$, which is $\mathrm{PGL}(n+1)$ in the Lie case and $\mathrm{PGL}(n)$ in the Jordan case, and the cyclic group of order $4$. In the Lie superalgebra case this can be found in \cite[Theorem 1]{serganova}.
\end{remark}

\begin{cor}[of the proof]\label{cara-do-isomorfismo}
Let $\varphi\colon Q\rightarrow Q$ be an automorphism. Then there is a scalar $\lambda\in \mathbb{F}^\times$ such that
$\varphi_\bo(\munderbar{x})=\lambda \underline {\varphi_\bz(x)}$ for all $x\in Q^\bz$. In fact, $\lambda^2=1$ if $\varphi_\bz$ is inner and $\lambda^2=-1$ otherwise.\qed
\end{cor}

\section{Graded modules over graded Lie or Jordan algebras}\label{se:graded-modules}

Our approach to gradings on $Q$ will be based on considering $Q^\bo$ as a graded $Q^\bz$-module. We will now review some basic concepts regarding graded modules. A comprehensive treatment of finite-dimensional graded modules over semisimple Lie algebras can be found in \cite{EK15}.

\begin{defi}
Let $G$ be an abelian group. By a \emph{$G$-graded vector space} we mean simply a vector space $V$ together with a vector space decomposition $\Gamma:\,V=\bigoplus_{g\in G} V_{g}$.
If $V=\bigoplus_{g\in G} V_{g}$ and $W\bigoplus_{g\in G} W_{g}$ are two graded vector spaces and $T:V\rightarrow W$ is a linear map, we say that $T$ is \emph{homogeneous of degree $h$}, for some $h\in G$, if $T(V_g)\subseteq W_{hg}$ for all $g\in G$.
\end{defi}

If $V$ is a finite-dimensional $G$-graded vector space, we obtain an induced $G$-grading on the associative algebra $\End(V)$ and, since $G$ is abelian, on the Lie algebra $\gl(V)=\End(V)^{(-)}$ and the Jordan algebra $\End(V)^{(+)}$.

\begin{defi}
Let $A= \bigoplus_{g\in G} A_g$ be a $G$-graded algebra (associative, Lie or Jordan) and let $V = \bigoplus_{g\in G} V_g$ be an $A$-module that is also a $G$-graded vector space. We say that $V$ is a \emph{$G$-graded module over $A$} if $A_g\cdot V_h\subseteq V_{gh}$ for all $g,h\in G$.
\end{defi}

\begin{defi}
Let $V$ together with $\Gamma$ be a $G$-graded vector space and let $d$ be an element of $G$. We denote by $\Gamma^{[d]}$ the $G$-grading given by relabeling the component $V_g$ as $V_{gd}$, for all $g \in G$. This is called the \emph{(right) shift of $\Gamma$ by $d$}. The $G$-graded vector space given by $V$ together with this new grading is denoted by $V^{[d]}$.
\end{defi}

\begin{lemma}\label{lemmazero}
Let $V$ be a graded module over a $G$-graded algebra $A$. Then, for any $d \in G$, the shift $V^{[d]}$ is also a graded module over $A$.
\qed
\end{lemma}

The next result is the converse of the above lemma in a special case. Recall that a simple module is said to be \emph{absolutely simple} if it remains simple upon the extension of scalars to the algebraic closure.

\begin{prop}\label{lemmaone}
Let $A$ be a $G$-graded algebra and let $V$ be a finite-di\-men\-sional (ungraded) absolutely simple $A$-module. If $\Gamma$ and $\Gamma'$ are two $G$-gradings that make $V$ a graded module over $A$ then $\Gamma'$ is a shift of $\Gamma$.
\end{prop}

\begin{proof}
The Lie and Jordan cases reduce to the associative case by considering the universal enveloping algebra. 
So, let $A$ be an associative algebra.
Since $V$ is finite-dimensional, the Density Theorem implies that the representation $\rho\colon A\to\End(V)$ is surjective.
It will be convenient to denote $V$ equipped with the gradings $\Gamma$ and $\Gamma'$ by $V_{\Gamma}$ and $V_{\Gamma'}$, respectively.
Since $\rho$ is a surjective homomorphism of graded algebras, we have $\End(V_{\Gamma})_g=\rho(A_g)=\End(V_{\Gamma'})_g$ for all $g\in G$.
The result follows from the next lemma.
\end{proof}

\begin{lemma}
Let $V$ be a finite-dimensional vector space and let $\Gamma$, $\Gamma'$ be $G$-gradings on $V$ that induce the same grading on $\End(V)$. Then $\Gamma'$ is a shift of $\Gamma$.
\end{lemma}

\begin{proof}
Let $I$ be a minimal graded left ideal of $R=\End(V)$. By Lemma 2.7 of \cite{EKmon}, there exist $g,g' \in G$ such that $V_{\Gamma} \cong I^{[g]}$ and $V_{\Gamma'} \cong I^{[g']}$. Thus, we have an isomorphism $V_{\Gamma'} \rightarrow (V_{\Gamma})^{[g^{-1}g']}$ of graded $R$-modules.
Such an isomorphism must be a scalar operator and, therefore, it leaves all subspaces invariant. We conclude that $\Gamma'$ is a shift of $\Gamma$.
\end{proof}

\section{Gradings on $\Sl(n)$ and $M_n^{(+)}$}\label{gradings-on-sln}

Since we are going to reduce the classification of gradings on $Q$ to the same problem for $Q^\bz$, we need to recall some facts about gradings on the Lie algebra $\mathfrak{sl}(n)$ and Jordan algebra $M_n^{(+)}$. 
Gradings by abelian groups on any finite-dimensional algebra $A$ are controlled by the automorphism group scheme of $A$ (see e.g. \cite[\S 1.4]{EKmon}). If $\mathrm{char}\,\mathbb{F}\ne 2$ and $n\ge 3$ ($n>3$ if $\mathrm{char}\,\mathbb{F}=3$) then the automorphism group schemes of $\mathfrak{psl}(n)$ and $M_n^{(+)}$ are isomorphic (see \cite[\S 3.1 and \S 5.6]{EKmon}). 
So we concentrate on the Lie case, since the Jordan case is completely analogous. 

If a grading $\Gamma$ on $S=\mathfrak{sl}(n)$ is the restriction of a grading of $R=M_n$, we say that $\Gamma$ is a \emph{Type I} grading. Otherwise, we say that $\Gamma$ is a \emph{Type II} grading. If $\mathbb{F}$ is algebraically closed, Type I gradings are characterized by the property that the image of $\eta_\Gamma\colon\widehat{G}\to\mathrm{Aut}(S)$ consists of inner automorphisms of $S$. Type II gradings are related to Type I gradings in the following way.

\begin{defi}[\cite{BK10}]
If $\Gamma:\, S = \bigoplus_{g \in G} S_g$ is a $G$-grading of Type II on $S$, then there exists a unique element $h \in G$ of order $2$ such that the coarsening $\overline{\Gamma}$ induced by the quotient map $G\to\overline{G}=G/\langle h \rangle$
is a $\overline{G}$-grading of Type I (see \cite[\S 3.1]{EKmon}). 
Moreover, for any $\chi\in\widehat{G}$, the automorphism $\eta_{\Gamma}(\chi)$ is inner if and only if $\chi(h)=1$.
We call $h$ the \emph{distinguished element} associated to the grading $\Gamma$.
For a Type I grading, it is convenient to define $h=e$.
\end{defi}

The next two lemmas will be crucial for describing gradings on $Q$. First we recall the concept of tensor product of graded spaces.

\begin{defi}
Given two $G$-graded vector spaces $V=\bigoplus_{g\in G} V_g$ and $W=\bigoplus_{g\in G} W_g$, we define their tensor product to be the vector space $V\otimes W$ 
together with the $G$-grading given by $(V \otimes W)_g = \bigoplus_{ab=g} V_{a} \otimes W_{b}$.
\end{defi}

\begin{lemma}\label{jordan-product}
Let $\Gamma$ be a $G$-grading on $\mathfrak{sl}(n)$ and let $h$ be its distinguished element. If we extend $\Gamma$ to a grading of $\mathfrak{gl}(n)$ by declaring the identity matrix to have degree $h$ then the map 
\[
\begin{matrix}
  J : & \mathfrak{gl}(n) \otimes \mathfrak{gl}(n) & \rightarrow & \mathfrak{gl}(n)\\
      & x\otimes y              & \mapsto     & xy+yx\\
\end{matrix}
\]
is homogeneous of degree $h$.
\end{lemma}

\begin{proof}
Let $\Delta$ be the indicated extension of $\Gamma$ to the Lie algebra $\gl(n)$. If $\Gamma$ is a Type I grading then $\Delta$ is actually a grading on the associative algebra $R=M_n$ and hence the map $J$ is homogeneous of degree $h=e$.  

Now suppose that $\Gamma$ is a Type II grading and consider its extension $\Delta:\,R=\bigoplus_{g\in G}R_g$, which is a grading on the Lie algebra $\gl(n)$ but not on the associative algebra $M_n$. Without loss of generality, we may assume $\mathbb{F}$ algebraically closed. Let $\overline{G}=G/\langle h \rangle$ and consider the coarsening 
$\overline{\Delta}:\,R=\bigoplus_{\bar{g}\in\overline{G}}R_{\bar{g}}$
induced by the quotient map $G\to\overline{G}$, i.e., $R_{\bar{g}}=R_g\oplus R_{gh}$. If $x\in R_a$ and $y\in R_b$ for some $a,b\in G$ then 
\[
xy \in R_{\overline{ab}}=R_{ab} \oplus R_{abh},
\] 
since $\overline{\Delta}$ is a grading on the associative algebra $R$. 
Hence we can write $xy= z_0 + z_1$ with $z_0 \in R_{ab}$ and $z_1 \in R_{abh}$.
Pick a character $\chi$ of $G$ such that $\eta_{\Gamma}(\chi)$ is not an inner automorphism of $S$. 
Then $-\eta_{\Gamma}(\chi)$ is the restriction of some anti-automorphism $\varphi$ of $R$. 
Since $\chi(h)=-1$, we have $\eta_{\Delta}(\chi)(1_R)=-1_R$ and hence $\varphi=-\eta_{\Delta}(\chi)$.
We compute: 
\[\begin{split}
(-\chi(b)y)(-\chi(a)x) = \varphi(y)\varphi(x) = \varphi(xy) &= \varphi(z_0)+\varphi(z_1) \\&= -\chi(ab)(z_0 + \chi(h)z_1),
\end{split}
\]
hence $yx = -z_0 + z_1$, which implies $xy+yx = 2z_1 \in R_{abh}$, as required. 
\end{proof}

\begin{lemma}\label{lie-product}
Let $\Gamma$ be a $G$-grading on $M_n^{(+)}$ and let $h$ be its distinguished element. Then the map 
\[
\begin{matrix}
  L : & M_n^{(+)} \otimes M_n^{(+)} & \rightarrow & M_n^{(+)}\\
      & x\otimes y              & \mapsto     & xy-yx\\
\end{matrix}
\]
is homogeneous of degree $h$.
\end{lemma}

\begin{proof}
If $\Gamma$ is a Type I grading then $\Gamma$ is actually a grading on the associative algebra $R=M_n$ and hence the map $L$ is homogeneous of degree $h=e$.  

Now suppose that $\Gamma$ is a Type II grading. Let $\overline{G}=G/\langle h \rangle$ and consider the coarsening $\overline{\Gamma}:\,R=\bigoplus_{\bar{g}\in\overline{G}}R_{\bar{g}}$ induced by the quotient map $G\to\overline{G}$. If $x\in R_a$ and $y\in R_b$ for some $a,b\in G$ then 
\[
xy \in R_{\overline{ab}}=R_{ab} \oplus R_{abh},
\] 
since $\overline{\Gamma}$ is a grading on the associative algebra $R$. Hence we can write $xy= z_0 + z_1$ with $z_0 \in R_{ab}$ and $z_1 \in R_{abh}$. Pick a character $\chi$ of $G$ such that $\eta_{\Gamma}(\chi)$ is not an inner automorphism. Then $\eta_{\Gamma}(\chi)=\varphi$, an anti-automorphism of $R$. We compute: 
\[
(\chi(b)y)(\chi(a)x) = \varphi(y)\varphi(x) = \varphi(xy) = \varphi(z_0)+\varphi(z_1) = \chi(ab)(z_0 + \chi(h)z_1),
\]
hence $yx = z_0 - z_1$, which implies $xy-yx = 2z_1 \in R_{abh}$, as required. 
\end{proof}

\section{Gradings on $Q(n)$}\label{se:gradings-on-Qn}

Now we are going to classify group gradings on the Lie and Jordan superalgebras $Q=Q(n)$, $n\ge 2$, under the assumption $\mathrm{char}\,\mathbb{F}\ne 2$ and, in addition, $\mathrm{char}\,\mathbb{F}$ does not divide $n+1$ in the Lie case. It will be convenient to use the following notation.

Let $V$ and $W$ be vector spaces with $G$-gradings $\Gamma:\,V=\bigoplus_{g\in G} V_g$ and 
$\Delta:\,W=\bigoplus_{g\in G} W_g$.  
We will denote by $\Gamma \oplus \Delta$ the $G$-grading on the superspace $V\oplus W$, where $V$ is regarded as the even part and $W$ as the odd part, given by $(V\oplus W)_g=V_g\oplus W_g$.
 
Recall from Section \ref{se:the-lie-superalgebra-Qn} the isomorphism $Q^\bz\to Q^\bo$ of $Q^\bz$-modules, which we denoted by $x\mapsto\munderbar{x}$. Given a $G$-grading $\Gamma$ on the vector space $Q^\bz$, we denote by $\underline{\Gamma}$ the image of $\Gamma$ under this isomorphism, i.e., $\underline{\Gamma}$ is given by $Q^\bo_g=\{\munderbar{x}\;|\;x\in Q^\bz_g\}$ for all $g\in G$.

We are ready to describe all possible $G$-gradings on $Q$.

\begin{thm}\label{ida}
Consider the simple Lie or Jordan superalgebra $Q=Q(n)$, $n\geq 2$. 
Let $\Gamma$ be a $G$-grading on $Q^\bz$, for some abelian group $G$, and let $h\in G$ be the distinguished element associated to $\Gamma$. 
Then the $G$-gradings on $Q$ extending $\Gamma$ are precisely the gradings of the form $\Gamma \oplus \underline{\Gamma}^{[d]}$, where $d\in G$ is such that $d^2=h$.
\end{thm}

\begin{proof}
First we show that if $d^2=h$ then $\Gamma \oplus \underline{\Gamma}^{[d]}$ is, indeed, a grading on $Q$, i.e., 
the product $[,]\colon Q\otimes Q\to Q$ (respectively, $\circ\colon Q\otimes Q\to Q$) is a homogeneous map of degree $e$ with respect to this grading. 
It is sufficient to consider the restrictions $Q^\bz\otimes Q^\bz\to Q^\bz$, $Q^\bz\otimes Q^\bo\to Q^\bo$ and $Q^\bo\otimes Q^\bo\to Q^\bz$.

The case of $Q^\bz\otimes Q^\bz\to Q^\bz$ is clear since $\Gamma$ is a grading on the Lie (respectively, Jordan) algebra $Q^\bz$.
In the second case, we observe that $Q^\bo$ equipped with $\underline{\Gamma}$ is a graded $Q^\bz$-module (by the definition of $\underline{\Gamma}$) 
and then apply Lemma~\ref{lemmazero}. For the third case, we will use the realization of $Q$ given by \eqref{identification} or \eqref{J_identification}. 

In the Lie case, according to \eqref{star}, the map
$[,]\colon Q^\bo\otimes Q^\bo \rightarrow Q^\bz$ is the composition of the following four maps:
\[
Q^\bo\otimes Q^\bo\stackrel{\beta}{\longrightarrow} Q^\bz\otimes Q^\bz\hookrightarrow \frg\otimes\frg\stackrel{J}{\longrightarrow}\frg\stackrel{\trl}{\longrightarrow}Q^\bz,
\]
where $\beta$ is an isomorphism given by $\munderbar{x}\otimes \munderbar{y} \mapsto x\otimes y$, $\frg=\gl(n+1)$ and $J(x\otimes y)=xy+yx$. As in Lemma \ref{jordan-product}, we extend the grading $\Gamma$ to a grading on $\frg$ by declaring the identity matrix to have degree $h$. 
Then $J$ is a homogeneous map of degree $h$. Also, $\beta$ is homogeneous of degree $d^{-2}=h^{-1}$, while the inclusion $Q^\bz\hookrightarrow\frg$ and projection 
$\frg\stackrel{\trl}{\to}Q^\bz$ are homogeneous of degree $e$. It follows that $[,]\colon Q^\bo\otimes Q^\bo \rightarrow Q^\bz$ has degree $e$, as desired.

The Jordan case is easier. According to \eqref{J_star}, the map
$\circ\colon Q^\bo\otimes Q^\bo \rightarrow Q^\bz$ is the composition of the following two maps:
\[
Q^\bo\otimes Q^\bo\stackrel{\beta}{\longrightarrow} Q^\bz\otimes Q^\bz
\stackrel{L}{\longrightarrow}
Q^\bz,
\]
where $\beta$ is an isomorphism given by $\munderbar{x}\otimes \munderbar{y} \mapsto x\otimes y$ and $L(x\otimes y)=xy-yx$. 
By Lemma \ref{lie-product}, $L$ is a homogeneous map of degree $h$. Also, $\beta$ is homogeneous of degree $d^{-2}=h^{-1}$. It follows that $\circ\colon Q^\bo\otimes Q^\bo \rightarrow Q^\bz$ has degree $e$, as desired. 
 
It remains to prove that all extensions of the grading $\Gamma$ on $Q^\bz$ to the superalgebra $Q$ are of the indicated form. 
Since $Q^\bo$ has to be a graded $Q^\bz$-module, we can apply Proposition \ref{lemmaone} to conclude that the grading on $Q$ must have the form 
$\Gamma \oplus \underline{\Gamma}^{[d]}$ for some $d\in G$. 
The above calculation of the degree of the product $Q^\bo\otimes Q^\bo \rightarrow Q^\bz$ shows that $d^2=h$.
\end{proof}

\begin{remark}\label{bracket_semiinv}
In the case of the Lie superalgebra $Q(n)$, we have just observed the following general phenomenon. Let $L=L^\bz\oplus L^\bo$ be a classical simple Lie superalgebra over an algebraically closed field of characteristic $0$. If $L$ is not isomorphic to $D(2,1,\alpha)$, it is shown in the proof of Proposition 2.1.4 of \cite{artigokac} that, with fixed Lie bracket on $L^\bz$ and $L^\bz$-module structure on $L^\bo$, 
the space of symmetric maps $L^\bo\otimes L^\bo\to L^\bz$ that make $L$ a Lie superalgebra has dimension $1$. It follows that if $L^\bz$ and $L^\bo$ are given $G$-gradings such that $L^\bz$ is a graded algebra and $L^\bo$ is a graded $L^\bz$-module then the product $[,]\colon L^\bo\otimes L^\bo\to L^\bz$ is automatically a homogeneous map of some degree, which we computed to be $h$ in the case of the gradings $\Gamma$ and $\underline{\Gamma}$ on $Q^\bz$ and $Q^\bo$, respectively.
\end{remark}

Note that if $\Gamma$ is a Type I grading on $Q^\bz$ then we can always extend it to $Q$: for example, $\Gamma \oplus \underline{\Gamma}$ does the job. 
But our theorem also shows that this is not the case for Type II gradings:

\begin{cor}
If $G$ does not have elements of order $4$ then every $G$-grading on $Q$ restricts to a Type I grading on $Q^\bz$.
\end{cor}

\begin{proof}
If the grading on $Q^\bz$ is of Type II then the distinguished element $h$ has order $2$ and hence $d$ must have order $4$. 
\end{proof}

The element $d$ (or its inverse) plays the same role for a $G$-grading on $Q$ as the element $h$ for a grading on $Q^\bz$. 
Indeed, a character $\chi\in\widehat{G}$ acts on $Q$ as the automorphism $\varphi=\eta_{\Gamma\oplus\underline{\Gamma}^{[d]}}(\chi)$, which satisfies 
$\varphi_\bo(\munderbar{x})=\chi(d)\underline {\varphi_\bz(x)}$ for all $x\in Q^\bz$ (cf. Corollary \ref{cara-do-isomorfismo}). 
In particular, $\chi$ acts as an inner automorphism of $Q$ if and only if $\chi(d)=1$ (see Remark \ref{extensions-of-automorphisms}). 

The role of $h=d^2$ for a $G$-grading on $Q$ can be seen using the construction of $Q$, described in Section \ref{se:the-lie-superalgebra-Qn}, in terms of the 
simple associative superalgebra $A=R\oplus uR$ where $R=M_{n+1}$ in the Lie case and $R=M_n$ in the Jordan case. 
The inner automorphisms of $Q$ are induced by inner automorphisms of $A$; they form a normal subgroup of index $4$ in $\mathrm{Aut}(Q)$ if $\sqrt{-1}\in\mathbb{F}$.
But there is an intermediate subgroup of index $2$, which consists of the automorphisms of $Q$ induced by all automorphisms of $A$.
A character $\chi\in\widehat{G}$ acts as an automorphism in this intermediate subgroup if and only if $\chi(h)=1$. 
It is easy to see that the characters satisfying $\chi(h)=-1$ act as automorphisms of $Q$ induced by super-anti-automorphisms of $A$ (with the minus sign in the Lie case).

Now we are going to determine when two gradings on $Q$ are isomorphic. 

\begin{thm}\label{th:isomorphism} 
Consider the simple Lie or Jordan superalgebra $Q=Q(n)$, $n\ge 2$. Let $G$ be an abelian group, let 
$\Gamma$ and $\Delta$ be $G$-gradings on $Q^\bz$, and let $c,d\in G$ be such that $\Gamma \oplus\underline{\Gamma}^{[c]}$ and 
$\Delta \oplus\underline{\Delta}^{[d]}$ are gradings on $Q$. Then $\Gamma \oplus\underline{\Gamma}^{[c]}$ and 
$\Delta \oplus\underline{\Delta}^{[d]}$ are isomorphic if and only if $\Gamma$ and $\Delta$ are isomorphic and $c=d$.
\end{thm}

\begin{proof}
It will be convenient to give two names to the isomorphism $Q^\bz\to Q^\bo$, $x \mapsto \munderbar{x}$, according to what gradings we use. If we consider $\Gamma$ and $\underline{\Gamma}^{[c]}$ on $Q^\bz$ and $Q^\bo$, respectively, then we will call it $\beta_c$, since it will have degree $c$. Analogously, if we consider $\Delta$ and $\underline{\Delta}^{[d]}$, it will have degree $d$ and we will call it $\beta_d$.
  
Suppose $\varphi\colon Q\rightarrow Q$ is an automorphism sending the grading $\Gamma \oplus \underline{\Gamma}^{[c]}$ onto 
$\Delta \oplus \underline{\Delta}^{[d]}$. By Corollary \ref{cara-do-isomorfismo}, we have
\begin{equation}\label{eq:phi0-phi1}
\varphi_\bo\circ \beta_c = \lambda \beta_d\circ \varphi_\bz
\end{equation}
for some $\lambda\in\mathbb{F}^\times$. By our assumption on $\varphi$, both $\varphi_\bz$ and $\varphi_\bo$ have degree $e$ and, hence, Equation \eqref{eq:phi0-phi1} implies $c=d$.

Conversely, suppose $c=d$ and there exists an automorphism $\psi\colon Q^\bz\to Q^\bz$ sending $\Gamma$ onto $\Delta$. By Proposition \ref{Aut-Qn}, there exists an automorphism $\varphi\colon Q\to Q$ such that $\varphi_\bz=\psi$. Now Equation \eqref{eq:phi0-phi1} implies that $\varphi_\bo$ has degree $e$, i.e., sends 
$\underline{\Gamma}^{[c]}$ onto $\underline{\Delta}^{[c]}$.
\end{proof}

For algebraically closed $\mathbb{F}$, the classification of $G$-gradings on $Q^\bz$ up to isomorphism can be found in \cite{BK10} or \cite[\S 3.3]{EKmon} for the Lie case, where $Q^\bz=\Sl(n+1)$, and in \cite[\S 5.6]{EKmon} for the Jordan case, where $Q^\bz=M_n^{(+)}$ (which has the same classification of gradings as $\Sl(n)$ if $n\ge 3$). Together with our Theorems \ref{ida} and \ref{th:isomorphism}, this gives a classification of $G$-gradings on the Lie or Jordan superalgebra $Q(n)$ up to isomorphism. We can also obtain such a classification for the corresponding associative superalgebra $A=M_n\oplus uM_n$, improving the description given in \cite{BS}. The following result is valid over an arbitrary field and reduces the classification to the matrix algebra $M_n$, for which it is known over algebraically closed $\mathbb{F}$ (see \cite{BK10} or \cite[\S 2.3]{EKmon}). For a grading $\Gamma$ on $A^\bz$, it will be convenient to denote by $u\Gamma$ the corresponding grading on $A^\bo=uA^\bz$, i.e., $A^\bo_g=\{ux\;|\;x\in A^\bz_g\}$ for all $g\in G$.

\begin{thm}\label{th:assoc} 
Consider the simple associative superalgebra $A=R\oplus uR$ where $R=M_n$. Then any grading on $A$ by an abelian group $G$ is of the form $\Gamma\oplus u\Gamma^{[c]}$, where $\Gamma$ is a grading on $R$ and $c\in G$ is such that $c^2=e$. This gradings is isomorphic to $\Delta\oplus u\Delta^{[d]}$ if and only if $\Gamma\cong\Delta$ and $c=d$.
\end{thm}

\begin{proof}
We observe that $A^\bo$ is isomorphic to $R=A^\bz$ as an $R$-bimodule. Indeed, an isomorphism $A^\bz\to A^\bo$ is given by $x\mapsto ux$. Since $A^\bo$ is a simple $R$-bimodule (or, equivalently, $(R\otimes R^\mathrm{op})$-module), we can proceed as in the proofs of Theorems \ref{ida} and \ref{th:isomorphism}. The details are  left to the reader. 
\end{proof}

We would like to point out that the approach based on graded modules over $G$-graded Lie algebras can be used to classify $G$-gradings on other classical simple Lie superalgebras (cf. Remark \ref{bracket_semiinv}) and also on abelian extensions of Lie algebras. 

\begin{ex}\label{pae} 
Let $\frg$ be the semidirect sum of $\mathfrak{psl}(n)$ and its adjoint module. Then any grading on $\frg$ by an abelian group $G$ is of the form $\Gamma\oplus\Gamma^{[c]}$, where $\Gamma$ is a grading on $\mathfrak{psl}(n)$ and $c\in G$. This gradings is isomorphic to $\Delta\oplus\Delta^{[d]}$ if and only if $\Gamma\cong\Delta$ and $c=d$.
\end{ex}

In general, if $\frg$ is any finite-dimensional Lie algebra over an algebraically closed field of characteristic $0$, then $\frg=\frs\oplus\frr$ where $\frr$ is the solvable radical of $\frg$ and $\frs$ is a semisimple Levi subalgebra. If $\frg$ is graded by a group $G$ then $\frr$ is $G$-graded and $\frs$ can be chosen $G$-graded (see \cite{Gor}). Note that the isomorphism classes of $\frs$ as a graded algebra and $\frr$ as a graded $\frs$-module are uniquely determined. If $\frr$ is abelian then the classification of $G$-gradings on $\frg$ reduces to classifying $G$-gradings on $\frs$ and, for each such grading, classifying graded $\frs$-modules. The first problem has been solved for simple $\frs$ except for types $E_6$, $E_7$ and $E_8$ (see \cite{EKmon} and also \cite{EK_d4} for $D_4$) and the second problem for any semisimple $\frs$ and abelian $G$ (see \cite{EK15}), but the relevant invariants have been explicitly computed only for simple $\frs$ different from $E_6$ and $E_7$ (for $D_4$ and $E_8$, see the Appendix in \cite{EK_d4}). 

To conclude this paper, we will show that the classification of fine gradings up to equivalence on the Lie (respectively, Jordan) superalgebra $Q$ 
is the same as for the Lie algebra $Q^\bz=\Sl(n+1)$ (respectively, the Jordan algebra $Q^\bz=M_n^{(+)}$), which can be found in \cite{Eld10} or \cite[\S 3.3]{EKmon} (respectively, in \cite[\S 5.6]{EKmon}).
 
We will need the following notation. Let $G$ be an abelian group and let $h\in G$. We denote by $G[h^{1/2}]$ the abelian group generated by $G$ and a new element $d$ subject only to the relation $d^2=h$. We will also denote $d$ by $h^{1/2}$. 

\begin{thm}\label{th:fine-gradings}
Consider the simple Lie or Jordan superalgebra $Q=Q(n)$, $n\ge 2$. If $\Gamma$ is a fine grading on $Q^\bz$ with universal group $G$ then $\tilde{\Gamma}=\Gamma\oplus\underline{\Gamma}^{[h^{1/2}]}$ is a fine grading on $Q$ with universal group $G[h^{1/2}]$, and every fine grading on $Q$ has this form (if we use its universal group). Moreover, $\tilde{\Gamma}$ and $\tilde{\Delta}$ are equivalent if and only if $\Gamma$ and $\Delta$ are equivalent.
\end{thm}

\begin{proof}
If $\Gamma$ is fine then so is $\tilde{\Gamma}$, because the grading on $Q^\bo$ is determined by the grading on $Q^\bz$ up to a shift (see Theorem \ref{ida}). 
Note that $\supp(\tilde{\Gamma})=S_\bz\cup S_\bo$ (disjoint union) where $S_\bz=\supp(\Gamma)$ and $S_\bo=S_\bz h^{1/2}$.

Let $\Delta$ be another fine grading on $Q^\bz$. Clearly, if $\tilde{\Gamma}$ and $\tilde{\Delta}$ are equivalent by means of an automorphism
$\varphi\colon Q\to Q$ and a bijection $\alpha\colon\supp(\Gamma)\to\supp(\Delta)$, then $\Gamma$ and $\Delta$ are equivalent by means of $\varphi_\bz$ 
and $\alpha_\bz$ (the restriction of $\alpha$ to $S_\bz$). Conversely, suppose $\Gamma$ and $\Delta$ are equivalent. Let $G'$ be the universal group of $\Delta$ 
and let $h'$ be its distinguished element. 
Then there exists an isomorphism $\beta\colon G\to G'$ such that ${}^\beta\Gamma$ is isomorphic to $\Delta$. It follows that $\beta(h)=h'$ and, hence, we can extend $\beta$ to an isomorphism $\tilde{\beta}\colon G[h^{1/2}]\to G'[(h')^{1/2}]$ such that $\tilde{\beta}(h^{1/2})=(h')^{1/2}$. 
Then ${}^{\tilde{\beta}}\tilde{\Gamma}$ is isomorphic to $\tilde{\Delta}$ by Theorem \ref{th:isomorphism}, so $\tilde{\Gamma}$ and $\tilde{\Delta}$ are equivalent.

Finally, suppose $G'$ is any abelian group and we have a $G'$-grading on $Q$, which, according to Theorem \ref{ida}, we can write as $\Gamma'\oplus\underline{\Gamma'}^{[d']}$ for some $d'\in G'$ 
satisfying $(d')^2=h'$, where $h'$ is the distinguished element of $\Gamma'$. 
There exists a fine grading $\Gamma$ on $Q^\bz$, with universal group $G$, and a homomorphism $\beta\colon G\to G'$ 
such that $\Gamma'={}^\beta\Gamma$. It follows that $\beta$ sends the distinguished element $h$ of $\Gamma$ to $h'$ and, 
hence, we can extend $\beta$ to a homomorphism $\tilde{\beta}\colon G[h^{1/2}]\to G'$ such that $\tilde{\beta}(h^{1/2})=d'$.
Then we obtain $\Gamma'\oplus\underline{\Gamma'}^{[d']}={}^{\tilde{\beta}}\tilde{\Gamma}$ by definition of $\tilde{\Gamma}$.
Since we started with an arbitrary group grading on $Q$, 
it follows that every fine grading on $Q$ has the form $\tilde{\Gamma}$, for some fine grading $\Gamma$ on $Q^\bz$, 
and, moreover, $G[h^{1/2}]$ is the universal group of $\tilde{\Gamma}$.
\end{proof}

To determine the structure of the universal groups of fine gradings on $Q$, note that, since $h^2=e$, 
the group $G[h^{1/2}]$ is isomorphic to $G\times\mathbb{Z}_2$ if $h$ is a square in $G$ and $\overline{G}\times\mathbb{Z}_4$ otherwise, where $\overline{G}=G/\langle h\rangle$.
If $\Gamma$ is of Type I, we have $h=e$, so $\overline{G}=G$ and $G[h^{1/2}]\cong G\times\mathbb{Z}_2$. 
If $\Gamma$ is of Type II, the group $G$ is computed in \cite[\S 3.3]{EKmon} (for algebraically closed $\mathbb{F}$) in terms of the extension 
$\langle h\rangle\to G\to\overline{G}$, which splits if and only if $h$ is not a square in $G$. 
Since the orders of torsion elements of $G$ are divisors of $4$, it follows that $G[h^{1/2}]\cong\overline{G}\times\mathbb{Z}_4$.
To summarize,  
\[
\text{the universal group of $\tilde{\Gamma}$}\cong\left\{\begin{array}{ll}
G\times\mathbb{Z}_2 & \text{if $\Gamma$ is of Type I};\\
\overline{G}\times\mathbb{Z}_4 & \text{if $\Gamma$ is of Type II}.
\end{array}\right.
\]
For Type I, the group $G$ is the universal group of the corresponding fine grading 
on the associative algebra $R$, where $R=M_{n+1}$ in the Lie case ($n\ge 2$) and $R=M_n$ in the Jordan case ($n\ge 3$); 
it is given in \cite[\S 2.3]{EKmon} for all fine gradings.
For Type II, $\overline{G}$ is the universal group of the grading on $R$ 
corresponding to the Type I coarsening induced by the quotient map $G\to\overline{G}$; it is computed in \cite[\S 3.2]{EKmon}. Note that $n=2$ is exceptional in the Jordan case, since $M_2^{(+)}$ is isomorphic to the Jordan algebra of a nondegenerate bilinear form on the $3$-dimensional space, which admits two fine gradings, with universal groups $\mathbb{Z}_2^3$ and $\mathbb{Z}\times\mathbb{Z}_2$ (see \cite[\S 5.6]{EKmon}). Both are of Type II, so the universal groups of the corresponding gradings on $Q(2)$ are $\mathbb{Z}_2^2\times\mathbb{Z}_4$ and $\mathbb{Z}\times\mathbb{Z}_4$, respectively.

\bibliographystyle{amsplain}
\bibliography{refs} 

\end{document}